\documentclass[12pt,reqno]{amsart}
\usepackage{amssymb}
\usepackage{amsthm}
\usepackage{bm}
\usepackage{latexsym}
\usepackage{amsmath}
\usepackage{eufrak}
\usepackage{mathrsfs}
\usepackage{caption}
\usepackage{amscd}
\usepackage[all,cmtip]{xy}
\usepackage[usenames]{color}
\usepackage{graphicx}
\usepackage{color}
\usepackage{amscd}
\usepackage{float}
\usepackage{graphics}
\usepackage{tikz}
\usepackage{tikz-cd}
\usepackage{comment}
\usepackage{xspace}
\usepackage{mathtools}

\usepackage{amssymb}
\usepackage{amsthm}
\usepackage{graphicx}
\usepackage{subfig}
\usepackage{enumerate}
\usepackage{hyperref}

\usetikzlibrary{patterns,decorations.pathreplacing}
\usepackage{marginnote}
\theoremstyle{plain}

\newtheorem{theorem}{Theorem}[section]
\newtheorem{proposition}[theorem]{Proposition}
\newtheorem{lemma}[theorem]{Lemma}
\newtheorem{corollary}[theorem]{Corollary}
\newtheorem{conjecture}[theorem]{Conjecture}

\theoremstyle{definition}

\newcommand{\appsection}[1]{\let\oldthesection\thesection
\renewcommand{\thesection}{Appendix \oldthesection}
\section{#1}\let\thesection\oldthesection}

\newtheorem{definition}[theorem]{Definition}

\theoremstyle{remark}

\def\Z{{\mathbb{Z}}}

\def\C{{\mathbb{C}}}
\def\P{{\mathbb{P}}}

\def\HH{{\mathcal{H}}}
\def\OO{{\mathcal{O}}}

\def\LL{{\mathcal{L}}}

\def\AR{{\mathcal{A}}}

\begin{document}
\title{Savage Surfaces}
\author[Sergio Troncoso]{Sergio Troncoso}
\email{stroncoso1@mat.uc.cl}
\address{Facultad de Matem\'aticas, Pontificia Universidad Cat\'olica de Chile, Campus San Joaqu\'in, Avenida Vicu\~na Mackenna 4860, Santiago, Chile.}

\author[Giancarlo Urz\'ua]{Giancarlo Urz\'ua}
\email{urzua@mat.uc.cl}
\address{Facultad de Matem\'aticas, Pontificia Universidad Cat\'olica de Chile, Campus San Joaqu\'in, Avenida Vicu\~na Mackenna 4860, Santiago, Chile.}

\date{\today}

\begin{abstract} 
Let $G$ be the topological fundamental group of a given nonsingular complex projective surface. We prove that the Chern slopes $c_1^2(S)/c_2(S)$ of minimal nonsingular projective surfaces of general type $S$ with $\pi_1(S) \simeq G$ are dense in the interval $[1,3]$.
\end{abstract}

\maketitle

%----------------------------------------------------------------
\section{Introduction} \label{intro}

By the Lefschetz hyperplane theorem, we know that the fundamental group of any nonsingular projective variety is the fundamental group of some nonsingular projective surface. Groups that are fundamental groups of varieties are abundant. Serre proved, for example, that any finite group is realizable \cite{S58}. For singular surfaces we know that every finitely presented group is possible as fundamental group by \cite{KK14} (reducible surfaces) and \cite{K13} (irreducible surfaces), but there are of course various restrictions in the case of nonsingular projective surfaces. (See the survey \cite{A95} for more on that topic, and see the book \cite{ABCKT96} for K\"ahler manifolds.) A natural geographical question is: \textit{Are there any constraints for the Chern slope of surfaces of general type after we fix the fundamental group?} In more generality, this question has been studied for 4-manifolds (cf. \cite{KL09}) with a particular focus on symplectic 4-manifolds (see e.g. \cite{G95}, \cite{BK06}, \cite{BK07}, \cite{Park07}). For example, Park showed in \cite{Park07} that the set of Chern slopes $c_1^2/c_2$ of minimal symplectic 4-manifolds $S$ with $\pi_1(S) \simeq G$ is dense in the interval $[0,3]$, for any fixed finitely presented group $G$. 

\vspace{0.3cm}
 
For complex surfaces, we know that simply connected surfaces of general type have Chern slopes dense in $[1/5,3]$ (see \cite{P81,Ch87,PPX96,U10,RU15}), which is the largest possible interval by the Noether inequality $1/5(c_2-36) \leq c_1^2$ and the Bogomolov-Miyaoka-Yau inequality $c_1^2 \leq 3 c_2$ (cf. \cite{BHPV04}). (See \cite{U17} for an analogue geographical result for surfaces in positive characteristic.) In general, however, it is known that for low slopes we do have some constraints for the possible fundamental groups. For instance, from \cite{MP07} we deduce that if $S$ is a surface of general type with $c_1^2(S) < \frac{1}{3} c_2(S)$ and $\pi_1(S)$ finite, then the order of $\pi_1(S)$ is at most $9$. We would also like to mention Reid's conjecture: The fundamental group of a surface with $c_1^2 < \frac{1}{2} c_2$ is either finite or commensurable with the fundamental group of a compact Riemann surface (see \cite[p.294]{BHPV04} for details). Pardini's proof \cite{Par05} of the Severi inequality together with Xiao's result \cite[Theorem 1]{X87} give evidence on this conjecture at the level of \'etale fundamental groups. 

\vspace{0.3cm}

On the other hand, we remark that a similar question for pairs $(c_1^2,c_2)$ has much stronger constraints. By Gieseker \cite{G77}, there are only finitely many possibilities of $\pi_1$ for a given pair. A concrete example: It is expected that for numerical Godeaux surfaces (i.e., $c_1^2=1$, $p_g=q=0$) the fundamental group belongs to the set $\{ 1, \Z/2, \Z/3, \Z/4, \Z/5\}$. Also, on the Bogomolov-Miyaoka-Yau line we have that $\pi_1(S)$ is an infinite group (since those surfaces are ball quotients by results of Miyaoka \cite[Prop.2.1.1]{Mi84} and Yau \cite{Yau77}), and, on the opposite side, on the Noether line we have only simply connected surfaces by the classification of Horikawa \cite{Ho75,Ho76}. In this article we prove the following.

\begin{theorem}
Let $G$ be the (topological) fundamental group of a nonsingular complex projective surface. Then the Chern slopes $c_1^2(S)/c_2(S)$ of minimal nonsingular projective surfaces of general type $S$ with $\pi_1(S)$ isomorphic to $G$ are dense in the interval $[1,3]$.
\end{theorem}

In this way, for instance, any finite group $G$ densely populate the wide sector $[1,3]$. The method to prove the theorem is very different to the one used in \cite[Theorem 6.3]{RU15} for trivial $\pi_1$, but we do consider as key input the extremal simply connected surfaces constructed in that paper. An observation here is that the $\pi_1$ trivial surfaces constructed by Persson in \cite{P81} do not work for our method, and they cannot work since, if they do, then some of them would violate Mendes-Lopes--Pardini's theorem mentioned above for low Chern slopes. Chen surfaces in \cite[Theorem 1]{Ch87} do not work for our method either.

\vspace{0.3cm}

We now explain roughly the idea of the proof together with the central ingredients. Let $Y$ be a minimal nonsingular projective surface with $\pi_1(Y) \simeq G$, let $r \in [1,3]$, and let $\{X_p\}$ be a sequence of simply connected surfaces as in \cite[Theorem 6.3]{RU15}, so that $c_1^2(X_p)/c_2(X_p)$ approaches $r$ as $p \to \infty$. Let $\Gamma_p,B$ be very ample divisors in $X_p$ and $Y$ respectively, and consider the very ample divisor $\Gamma_p+B$ in $X_p \times Y$. As in \cite[Section 1]{catanese2000fibred}, one obtains a surface $S_p$ from the intersection of two general sections in $|\Gamma_p+B|$ so that $\pi_1(S_p) \simeq G$ (Lefschetz hyperplane theorem), but it is not possible to have the result for $c_1^2(S_p)/c_2(S_p)$ since we have no control on $\Gamma_p$. On the other hand, an appropriate $\Gamma_p$ to control $c_1^2(S_p)/c_2(S_p)$ may not be even ample, so we may not have $\pi_1(S_p) \simeq G$, or even an $S_p$ to start with. To overcome both difficulties, we consider a very special $\Gamma_p$ which works for $c_1^2(S_p)/c_2(S_p)$ and it is also a lef (\textit{Lefschetz effettivamente funziona}) line bundle, as introduced by de Cataldo and Migliorini \cite{de2002hard}. It turns out that such a $\Gamma_p$ allows us to prove existence of $S_p$ as above which, by a generalization of the Lefschetz hyperplane theorem due to Goresky and MacPherson \cite[Part II, Theorem 1.1]{goresky1988stratified}, satisfy $\pi_1(S_p) \simeq G$. These surfaces are used to prove the claim on density of Chern slopes in $[1,3]$. We also show that it is not possible to improve this lower bound $1$ by using modifications of the surfaces $X_p$.

\vspace{0.3cm}

We finish the paper with two conjectures in relation to geography of Chern slopes for surfaces with ample canonical class, and for Brody hyperbolic surfaces, which might be proved by using the same techniques as in this paper.

%%%%%%%%%%%%%%%%%%%%
\subsubsection*{Acknowledgments}
The authors thank Xavier Roulleau for the hospitality with the first author during his stay in Universit\'e d'Aix-Marseille. The authors have benefited from conversations with Feng Hao, Margarida Mendes Lopes, Luca Migliorini, Rita Pardini, Jongil Park, Ulf Persson, Francesco Polizzi, and Matthew Stover. We are grateful to all of them. We thank the referee for useful comments. The first author was partially funded by the Conicyt Doctorado Nacional grant 2017/ 21170644. The second author  was supported by the FONDECYT regular grant 1190066.

%----------------------------------------------------------------
\section{Semi-small morphisms, lef line bundles, Bertini and Lefschetz type theorems} \label{s1}

Throughout this paper the ground field is $\C$. For a given line bundle $M$ and integer $n$, the line bundle $M^{\otimes n}$ will be denoted by by either $nM$ or $M^n$. The following definition can be found in several places, e.g. 
\cite[p.151]{goresky1988stratified}, \cite[Def. 4.1]{Mig95} or \cite[Def. 2.1.1]{de2002hard}.

\begin{definition}\label{ss}
Let $X,Y$ be irreducible varieties. For a proper surjective morphism $f \colon X \to Y$, we define 
$$Y_f^k=\{y\in Y| \operatorname{dim}f^{-1}(y)=k\}.$$ We say that $f$ is \textit{semi-small} if $\operatorname{dim}(Y_f^k)+2k\leq \operatorname{dim}X$ for every $k\geq 0$. (Note that $\operatorname{dim}(\emptyset)=-\infty.$) If no confusion can arise, the subscript $f$ will be suppressed.
\end{definition}

We note that for a semi-small morphism we have $\operatorname{dim}(X)=\operatorname{dim}(Y)$.
 
\begin{lemma}\label{surjectivess}
Let $X,Y$ be surfaces. If $f \colon X \to Y$ is a proper surjective morphism, then $f$ is semi-small.
\end{lemma}

\begin{proof}
It is clear that $\operatorname{dim}(Y^1)=0$ and
$\operatorname{dim}(Y^0)=2,$ since $f$ is surjective. Then
the inequality  $\operatorname{dim}(Y^k)+2k\leq \operatorname{dim}(X)$ holds
for any $k\geq0.$
\end{proof}

\begin{proposition}\label{Productoss}
Let $f\colon X \to Y$ and $g\colon Z \to W$ be two semi-small morphisms. Then the product morphism $f\times g \colon X\times Z\to Y\times W$ is a semi-small morphism.
\end{proposition}

\begin{proof}
Let $n=\operatorname{dim}(X)$ and $m=\operatorname{dim}(Z)$. Since $f$ and $g$ are semi-small, then we have that $\operatorname{dim}(Y^k)\leq n-2k$ for any $k\geq 0$, $\operatorname{dim}(Z^l)\leq m-2l$ for any $l\geq 0$, and $\operatorname{dim}(Y^0)=n,\operatorname{dim}(W^0)=m$. We also have $(Y\times W)^q=\displaystyle\bigcup_{i+j=q} Y^i \times W^j$, and so $$\operatorname{dim}(Y \times W)^q \leq  \max_{\substack{i+j=q}} \operatorname{dim}(Y^i \times W^j) \leq  n+m-2i-2j = n+m-2q .$$ Hence $f \times g$ is semi-small.
\end{proof}

\begin{proposition}\label{finitess}
Let $X,Y,Z$ be nonsingular projective varieties. Assume that  $f \colon X \to Y$ is semi-small, and that $g \colon Y\to Z$ is finite morphism. Then, $h= g\circ f \colon X\to Z$ is semi-small.
\end{proposition}

\begin{proof}
Since $g$ is a finite, we have $Z_h^k=g(Y_f^k)$ for each $k\geq 0$, and so $\operatorname{dim}(Z_h^k)=\operatorname{dim}(Y_f^k)$. Thus $\operatorname{dim}(Z_h^k)+2k\leq \operatorname{dim}(X)$, and so $h$ is  semi-small.
\end{proof}

\begin{definition} (\cite[Def. 2.3]{de2002hard}) 
Let $X$ be a nonsingular projective variety, and let $M$ be a line bundle on $X$. We say that $M$ is \textit{lef}  if there exists $n>0$ such that $|nM|$ is generated by global sections, and the morphism $\psi_{|nM|}$ associated to $|nM|$ is semi-small onto its image. 
The \textit{exponent} of $M$ is the smallest
$n$ so that $M$ is lef. We denote it by $\operatorname{exp}(M)$.
\label{llb}
\end{definition}

If $L$ is an ample line bundle, then $L$ is lef. If moreover $L$ is very ample, then $\operatorname{exp}(L)=1$. Next we write a corollary of Proposition \ref{finitess} which will be used later. 

\begin{proposition}\label{Llef}
Let $f \colon X\to Y$ be semi-small between nonsingular projective varieties, and let $L$ be very ample on $Y$.  Then $f^*(L)$ is lef with $\operatorname{exp}(f^*(L))=1$. 
\end{proposition}

A useful Bertini type theorem for lef line bundles is the following. (See \cite[Prop. 2.1.7]{de2002hard} or \cite[Lemma 4.3]{Mig95}.)
  
\begin{theorem}\label{bertini}
Let $X$ be a nonsingular projective variety of dimension at least $2.$ Let $M$ be a lef line bundle on $X.$ Assume that $M$ is globally generated and with $\operatorname{exp}(M)=e$. Then any generic member $Y \in |M|$ is a nonsingular projective variety, and the restriction $M|_Y$ is lef on $Y$ with $\operatorname{exp}(M|_Y)\leq e$.
\end{theorem}

We now state a Lefschetz type theorem relevant for the computation of the fundamental group, which is due to Goresky and MacPherson \cite{goresky1988stratified}, and was conjectured by Deligne \cite{D79}. For comparison, we mention the usual Lefschetz theorem for ample line bundles. (See e.g. See \cite[Theo. 3.1.21]{Laz17}.)

\begin{theorem}[Lefschetz theorem for homotopy groups]\label{clefs}
Let $X$ be a nonsingular projective variety of dimension $n$.
Let $\iota \colon A \to X$ be the inclusion of an effective ample divisor $A$. Then the induced homomorphism $$\iota^*\colon \pi_i(A) \to \pi_i(X)$$ is bijective if $i\leq n-2$, and it is surjective if $i=n-1$.
\end{theorem} 

\begin{theorem}\label{gm}
Let $X$ be a nonsingular projective variety of dimension $n$. Suppose that $f \colon X \to \mathbb{P}^N$ is a proper morphism, and let $H$ be a linear subspace of
codimension $c$. Define $\phi(k):=\operatorname{dim}\left((\mathbb{P}^N\setminus H)_f^k\right)$. Then the induced homorphism $$\pi_i(f^{-1}(H)) \to \pi_i(X)$$ is an isomorphism if $i<\hat{n}$, and it is surjective if $i=\hat{n}$, where $$\hat{n}=n-1-\sup_{k}(2k-n+\phi(k)+\inf_k(\phi(k),c-1)).$$
\end{theorem}

\begin{proof}
This is \cite[Part II, Theorem 1.1]{goresky1988stratified} pages 150-151, under the hypothesis that $f$ is proper.
\end{proof}

\begin{corollary}
If $H$ is a hyperplane in Theorem \ref{gm} and $f \colon X \to \mathbb{P}^N$ is semi-small into its image, then 
$$\pi_i(f^{-1}(H)) \simeq \pi_i(X)$$ if
$i< n-1$. 
\end{corollary}

\begin{proof}
The case $f(X) \subset H$ is trivial. In the computation of $\hat{n}$ we can ignore the values $\phi(k)=-\infty$. Then we compute $\hat{n}=n-1$ since dim$(f(X))=n$, the codimension of $H$ is $c=1$, and we have the inequality $\phi(k)\leq \operatorname{dim} ((f(X))^k_f)$. The last inequality is because $(\P^N \setminus H)^k_f=(f(X) \setminus (f(X) \cap H))^k_f \subset f(X)^k_f$.
\end{proof}

\begin{corollary}\label{fundamentallef}
Let $X$ be a nonsingular projective variety with $\operatorname{dim}(X)$ $\geq 3$. Let $M$ be a lef line bundle on $X$ with $\operatorname{exp}(M)=1$. If $E \in |M|$, then 
$\pi_1(E)\simeq \pi_1(X)$.
\end{corollary}

\begin{corollary}\label{Flef} 
Let $X$ be a nonsingular projective variety with $\operatorname{dim}(X)$ $\geq 4$. Let $M$ be a lef line bundle  with $\operatorname{exp}(M)=1$. Then a generic member $E \in |M|$ is nonsingular projective variety, and $M_E:=M|_E$ is lef. Moreover, if $F\in |M_E|$, then $\pi_1(F)\simeq \pi_1(X)$.
\end{corollary}

\begin{proof}
The first part is just Theorem \ref{bertini}. If $F\in M_E$, then by Corollary \ref{fundamentallef} we obtain that $\pi_1(F)\simeq \pi_1(E)\simeq \pi_1(X)$.
\end{proof}

%---------------------------------------------------------
\section{RU surfaces} \label{s2}

In this section we recall some surfaces of general type $X_p$ from \cite[Section 6]{RU15} which are keys in the main result of this paper. We will follow the conventions in \cite{RU15}. In particular, an arrangement of curves is a collection of curves $\{C_1, \ldots, C_r \}$ on a nonsingular surface. A $k$-point of an arrangement of curves is a point of it locally of the form $(0,0) \in \{(x-\xi_1 y) \cdots (x- \xi_k y)=0 \} \subset \C_{x,y}^2$ for some $\xi_i \neq \xi_j$.

Let $p \geq 5$ be a prime number, and let $\alpha>0, \beta > 0$ be integers. Let $n=3 \alpha p$. Let $\tau \colon H \rightarrow \P^2$ be the blow-up at the $12$ $3$-points of the dual Hesse arrangement of $9$ lines $(x^3-y^3)(y^3-z^3)(x^3-z^3)=0$ in $\P^2$. As defined in \cite[Sections 3 and 5]{RU15}, we will consider the diagram of varieties and morphisms (where $i \in \{0,1,\infty,\zeta \}$) $$ \xymatrix{ Y_n \ar[r]^{\sigma_n} & Z_n \ar[r]^{\varphi_n} & H  \ar[r]^{ \tau} \ar[d]^{\pi'_i} & \P^2 \\ & & \P^1 &  }$$ The three singular fibers of $\pi'_i$ are denoted by $F_{i,1}, F_{i,2}, F_{i,3}$. Each $F_{i,j}$ consists of four $\P^1$'s: one central curve $N_{i,j}$ with multiplicity $3$, and three reduced curves transversal to $N_{i,j}$ at one point each. We write $N_i=N_{i,1}+N_{i,2}+N_{i,3}$. Let $M$ be the $9$ $\P^{1}$'s from the lines of the dual Hesse arrangement, and let $N$ be the $12$ exceptional $\P^{1}$'s from its twelve $3$-points. We have $N=\sum_{i=0,1,\zeta,\infty}N_{i}$, and $$F_{i,1}+F_{i,2}+F_{i,3}=M+3N_{i}.$$

We now consider the very special arrangement of $\frac{4n^2-12}{3}$ elliptic curves ${\HH}'_n =\mathcal{E}_{0} + \mathcal{E}_{1} + \mathcal{E}_{\infty} + \mathcal{E}_{\zeta}$ in $H$. For $i\in \{0,1,\infty,\zeta\}$, let $\mathcal{E}'_i$ be $\beta^2 p^2$ general fibers of $\pi'_i$ (defined also in \cite[Section 3]{RU15}), and let $\AR_{2d}=L_1+ \ldots + L_{2d}$ be the strict transform of an arrangement of $2d$ general lines in $\P^2$, where $3 \leq 2d \leq p$. We define $a_0=a_1=b_i=1$ for $1\leq i \leq d$, and $a_{\infty}=a_{\zeta}=b_i=p-1$ for $d+1 \leq i \leq 2d$. Then $$ \OO_H \Big( \sum_{i=0,1,\zeta,\infty} 3 a_i \mathcal{E}_i + \sum_{i=0,1,\zeta,\infty} 3 a_i \mathcal{E}'_i + \sum_{i=0,1,\zeta,\infty} a_i (F_{i,1}+F_{i,2}+F_{i,3})+ \sum_{i=1}^{2d} 3 b_i L_i \Big)$$ is isomorphic to $\LL_0^{p}$ where $$\LL_0:= \OO_H \Big( 3p(3\alpha^2+\beta^2)\big(\sum_{i=0,1,\zeta,\infty} a_i F_i \big)+3dL \Big),$$ and all symbols have been defined in \cite[Section 5]{RU15}. For each $i$, we denote the strict transform of $\mathcal{E}_i$, $\mathcal{E}'_i$, $L_j$, $F_{i,j}$ in $Z_n$ by the same symbol, where $\varphi_n \colon Z_n \to H$ is the blow-up of $H$ at all the $\frac{(n^2-3)(n^2-9)}{3}$ $4$-points in ${\HH}'_n$. Then $$\OO_{Z_n} \big( \sum_{i=0,1,\zeta,\infty} 3 a_i \mathcal{E}_i + \sum_{i=0,1,\zeta,\infty} 3 a_i \mathcal{E}'_i + \sum_{i=0,1,\zeta,\infty} a_i (F_{i,1}+F_{i,2}+F_{i,3})+ \sum_{i=1}^{2d} 3 b_i L_i  \big)$$ is $\LL_1^{p}$ where  $\LL_1:= \varphi_n^*(\LL_0) \otimes \OO_{Z_n}(-6E)$, and $E$ is the exceptional divisor of $\varphi_n$. Again, we denote the strict transform of $\mathcal{E}_i$, $\mathcal{E}'_i$, $L_j$, $F_{i,j}$, $M$, $N_i$, $N$ in $Y_n$ by the same symbol, where $\sigma_n \colon Y_n \to Z_n$ is the blow-up at all the $4(n^2-3)$ $3$-points in ${\HH}'_n$. Then we have $$\OO_{Y_n} \Big( \sum_{i=0,1,\zeta,\infty} 3 a_i \mathcal{E}_i + \sum_{i=0,1,\zeta,\infty} 3 a_i \mathcal{E}'_i + \sum_{i=0,1,\zeta,\infty} 3 a_i N_i+ \sum_{i=1}^{2d} 3 b_i L_i  \Big) \simeq \LL^{p}$$ where $\LL:= \sigma_n^*(\LL_1) \otimes \OO_{Y_n}(-2M-6G)$. 

With this data, we construct a $p$-th root cover of $Y_n$ branch along $$A:= \sum_{i=0,1,\zeta,\infty} \mathcal{E}_i + \sum_{i=0,1,\zeta,\infty} \mathcal{E}'_i + \sum_{i=0,1,\zeta,\infty} N_i+ \sum_{i=1}^{2d} L_i.$$ Let $f \colon X_p \to Y_n$ be the corresponding morphism for the $p$-th root cover, as in \cite[Section 5]{RU15}. The nonsingular projective surface $X_p$ is simply connected \cite[Prop.6.1]{RU15}, and minimal \cite[Prop.6.2]{RU15}.

Let us write $$A= \sum_j \nu_j A_j= \sum_{i=0,1,\zeta,\infty} 3 a_i \mathcal{E}_i + \sum_{i=0,1,\zeta,\infty} 3 a_i \mathcal{E}'_i + \sum_{i=0,1,\zeta,\infty} 3 a_i N_i+ \sum_{i=1}^{2d} 3 b_i L_i$$ where $A_j$ are the irreducible curves in $A$. Hence $\nu_j$ is equal to either $3a_i$ or $3b_k$ for some $i,k$. The arrangement $A$ has only $2$-points, and the number of these points is $$t_2 = 108 \alpha^2 \beta^2 p^4 +18 \beta^4 p^4+ 72d \alpha^2 p^2-25d+24d \beta^2p^2+2d^2.$$ By \cite[Prop.4.1]{RU15}, the log Chern numbers of $A$ are $$\bar{c}_1^2= n^4 +2t_2 -10d -48 \ \ \text{and} \ \ \bar{c}_2= \frac{n^4}{3} +t_2 -4d -12.$$

As in \cite[Section 5]{RU15}, the Chern numbers of $X_p$ are $$ c_1^2(X_p) = p \, \bar{c}_1^2 - 2\Big(t_2 + 2 \sum_j (g(A_j)-1) \Big) + \frac{1}{p} \sum_j A_j^2 - \sum_{i<j} c(q_{i,j},p) A_i \cdot A_j$$ and $$c_2(X_p) = p \, \bar{c}_2 - \Big(t_2 + 2 \sum_j (g(A_j)-1) \Big) + \sum_{i<j} l(q_{i,j},p) A_i \cdot A_j$$ where $0<q_{i,j}<p$ with $\nu_i + q_{i,j} \nu_j \equiv 0$ (mod $p$),  $$c(q_{i,j},p):=12s(q_{i,j},p)+l(q_{i,j},p),$$ and $s(q_{i,j},p)$ and $l(q_{i,j},p)$ are the Dedekind sum and the length of the Hirzebruch-Jung continued fraction associated to the pair $(q_{i,j},p)$ respectively (see \cite[Definition 5.2]{RU15}).

For the particular multiplicities $a_0=a_1=b_i=1$ for $1\leq i \leq d$ and $a_{\infty}=a_{\zeta}=b_i=p-1$ for $d+1 \leq i \leq 2d$ we chose, we have 
to consider only the numbers $c(p-1,p)=\frac{2p-2}{p}$ and $c(1,p)=\frac{p^2-2p+2}{p}$, and $l(p-1,p)=p-1$ and $l(1,p)=1$. Therefore, $$\sum_{i<j} c(q_{i,j},4p) A_i \cdot A_j=\frac{(2p-2)}{p} t_{2,1}+ \frac{(p^2-2p+2)}{p} t_{2,2}$$ and $$\sum_{i<j} l(q_{i,j},4p) A_i \cdot A_j= (p-1) t_{2,1}+ t_{2,2}$$ where $t_{2,1}$ and $t_{2,2}$ are the number of $2$-points corresponding to the singularities $\frac{1}{p}(1,p-1)$ and $\frac{1}{p}(1,1)$ respectively. Hence $$t_{2,1}=6 \beta^4 p^4+36 \alpha^2 \beta^2 p^4+36 d \alpha^2 p^2-13 d+12 d \beta^2 p^2+ d^2$$ and $$t_{2,2}=12\beta^4 p^4+ 72 \alpha^2 \beta^2 p^4 +36 d \alpha^2 p^2-12 d+12 d \beta^2 p^2+ d^2.$$ By plugging in the formulas for Chern numbers, we obtain that 
$$ c_1^2(X_p)= (81 \alpha^4 +144 \alpha^2 \beta^2 + 24\beta^4) p^5 + l.o.t.$$ and $$c_2(X_p)= (27 \alpha^4 + 144 \alpha^2 \beta^2 + 24 \beta^4) p^5 + l.o.t.,$$ where l.o.t. (lower order terms) is a Laurent polynomial in $p$ of degree less than $5$. In this way, we obtain that $$\text{lim}_{p \to \infty} \ \frac{c_1^2(X_p)}{c_2(X_p)} = \frac{27x^4+48x^2+8}{9x^4+48x^2+8}=:\lambda(x)$$ where $x:=\alpha/\beta$. We note that $\lambda\big([0,\infty^+]\big) =[1,3]$. This allows to prove the following theorem (see \cite[Theorem 6.3]{RU15}).

\begin{theorem}
For any number $r \in [1,3]$, there are simply connected minimal surfaces of general type $X$ with $c_1^2(X)/c_2(X)$ arbitrarily close to $r$.
\label{dense}
\end{theorem}

\begin{proposition} \label{rhcc}
Let $\Gamma_p:=f^*(L)$, where as before $L$ is the pull-back in $Y$ of a general line in $\P^2$. Then we have $\Gamma_p^2=p$ and $\Gamma_p \cdot K_{X_p}=-3p+(p-1)(2d+36\alpha^2p^2-12+12\beta^2p^2)$.
\end{proposition}

\begin{proof}
As $f$ is a generically finite morphism of degree $p$, we have $\Gamma_p^2=p$. Let us consider $L$ generic, so that $f^*(L)$ is a nonsingular projective curve. We note that $L \cdot N_i=0$ for all $i$, $L \cdot \sum_{i=1}^{2d} L_i=2d$, $L \cdot \sum_{i=0,1,\zeta,\infty} \mathcal{E}_i = 36 \alpha^2 p^2 -12$, and $L \cdot \sum_{i=0,1,\zeta,\infty} \mathcal{E}'_i = 12 \beta^2 p^2$. Therefore, the morphism $f_{\Gamma_p} \colon \Gamma_p \to L=\P^1$ is totally ramified at $2d+36\alpha^2p^2-12+12\beta^2p^2$ points, and so, by the Riemann-Hurwitz formula and adjunction, we obtain the desired equality for $\Gamma_p \cdot K_{X_p}$.
\end{proof}

\bigskip

We finish this section with a proof that the best lower bound for Chern slopes in this construction is indeed $1$. As it was shown above, the values of the $b_i$'s do not contribute in the asymptotic final result. We also point out that it is enough to have either $\sum_{i=0,1,\zeta,\infty} a_i = p$ or $\sum_{i=0,1,\zeta,\infty} a_i=2p$ by considering $0<a_i<p$ and multiplying by units modulo $p$. In fact, we can and do take $a_0=1$, $a_1=a$, $a_{\zeta}=b$, and $a_{\infty}=c$ with $1+a+b+c=mp$ for $m$ either equal to $1$ or $2$.

Through the formulas obtained above, we have $$ \lim_{x\to 0}\frac{c_1^2(X_p)}{c_2(X_p)}= \frac{12-\frac{1}{p}C}{6+\frac{1}{p}L} $$ where $C:=c(-a,p)+c(-b,p)+c(-c,p)+c(-ba^{-1},p)+c(-ca^{-1},p)+c(-cb^{-1},p)$, $L:=l(-a,p)+l(-b,p)+l(-c,p)+l(-ba^{-1},p)+l(-ca^{-1},p)+l(-cb^{-1},p)$, and all the $q$'s in these expressions are taken modulo $p$ with $0<q<p$. For example, for generic $a,b,c$ one can prove that $C/p$ and $L/p$ tend to $0$ as $p$ approaches infinity, and so the limit of the Chern slopes is $2$ (see \cite{U10} for these generic behaviour).

Since $c(q,p)=12s(q,p)+l(q,p)$, it is enough to show that $$6S+L \leq 3p + 3 - \frac{6}{p} $$ any $p$, where $S:= s(-a,p)+s(-b,p)+s(-c,p)+s(-ba^{-1},p)+s(-ca^{-1},p)+s(-cb^{-1},p)$. The proof will use the following numerical lemma.

\begin{lemma}
Let $0<q<p$ be coprime integers. Let $\frac{p}{q}=[e_1,\ldots,e_l]$. Then $\sum_{i=1}^l (e_i-1) \leq p-1$. 

\label{behavior}
\end{lemma}

\begin{proof}
We do induction on $p$. Say for all coprime pairs $(q',p')$ with $p'<p$ we have that the statement is true. We write $ \frac{p}{q} = [e_1,\ldots,e_l]$. Then $e_1=[p/q]+1$, and $\frac{q}{r} = [e_2,\ldots, e_l]$ with $(r,q)$ coprime and $q<p$. Hence $$\sum_{i=1}^l (e_i-1) = [p/q] + \sum_{i=2}^l (e_i-1) \leq [p/q] + q-1$$ by the induction hypothesis. Therefore, we should prove that $[p/q] + q \leq p$. Let $q \neq 1$ (otherwise we are done). Let $1 \leq r < q$ be the unique integer such that $[p/q]q+r=p$. Then $[p/q] + q \leq p$ is equivalent to $ \frac{q-r}{q-1} + q \leq p$. But $\frac{q-r}{q-1} \leq 1$ if $r \geq 1$, and $q+1 \leq p$. 
\end{proof}

\begin{proposition}
We have $6S+L \leq 3p + 3 - \frac{6}{p}$.
\label{lomaschico}
\end{proposition}

\begin{proof}
Let $0<q<p$ integers where $p$ is a prime number. Then (see e.g. \cite[Example 3.5]{U10}) $$12 s(q,p)= \frac{q+q^{-1}}{p} + \sum_{i=1}^l(e_i-3)$$ where $\frac{p}{q}=[e_1,\ldots,e_l]$ and $q^{-1}$ is the integer between $0$ and $p$ such that $q q^{-1} \equiv 1$ modulo $p$. Hence $6s(q,p)+l= \frac{q+q^{-1}}{2p} + \frac{1}{2} \sum_{i=1}^l(e_i-1)$. We note that always $\frac{q+q^{-1}}{2p} \leq \frac{p-1}{p}$. We now run this equality for each of the terms in $S$ and in $L$, and use Lemma \ref{behavior} to conclude that $$6S+L \leq 3p -3 + 6\frac{(p-1)}{p} = 3p + 3 - \frac{6}{p}.$$
\end{proof}

%--------------------------------------------------------------------------------------------------
\section{Key construction and density theorem}  \label{s2}

In this section, we generalize the construction used in \cite[Section 1]{catanese2000fibred} in the context of lef line bundles, which will be used for the main theorem. 

\begin{proposition} \label{Mlef}
Let $X$ and $Y$ be nonsingular projective surfaces. Let $p \colon X\times Y \to X$ and $q \colon X \times Y \to Y$ be the usual projections. Let $\Gamma$ and $B$ be lef line bundles on $X$ and $Y$ respectively. Assume that $\operatorname{exp}(\Gamma)=\operatorname{exp}(B)=1$. Then $p^*(\Gamma) \otimes q^*(B)$ is a lef line bundle on $X\times Y$ of exponent $1$.
\end{proposition}

\begin{proof}
This is elementary, we briefly give an argument. Let $M:=p^*(\Gamma)\otimes q^*(B)$. Let $s_0,\dots,s_l$ be a basis of $H^0(X,\Gamma)$, and let $t_0,\dots,t_b$ be basis of $H^0(Y,B)$. Since $H^0(X,\Gamma)\otimes H^0(Y,B) \simeq H^0(X\times Y, M)$ (see e.g. \cite[Fact III.22, i]{beauville1996complex}), then $M$ is generated by the global sections $s_it_j$ with $0\leq i\leq l$ and $0\leq j\leq b$. The morphism $\psi_{|M|} \colon X\times Y \to \mathbb{P}(|M|)$ is $\Sigma_{l,b}\circ (\psi_{|\Gamma|}\times\psi_{|B|})$, where $\Sigma_{l,b}$ is the Segre embedding. Therefore $\psi_{|M|}$ is semi-small into its image as $\psi_{|\Gamma|}\times\psi_{|B|}$ is semi-small by Proposition \ref{Productoss}. It follows that $M$ is lef and $\operatorname{exp}(M)=1$. 
\end{proof}

\begin{theorem}\label{ct}
Let $X$ and $Y$ be nonsingular projective surfaces with nef canonical class, and $K_X^2>0$. Let $B$ be a very ample line bundle on $Y$, and let $\Gamma$ be a lef line bundle on $X$ with $\operatorname{exp}(\Gamma)=1$.
 
Then there exist a nonsingular projective surface $S \subset X \times Y$ with the following properties:
  
\begin{enumerate}
\item $\pi_1(S)\simeq \pi_1(X)\times\pi_1(Y)$.
       
\item The morphisms $p|_S \colon S\to X$ and $q|_S \colon S\to Y$ have degrees $\operatorname{deg}(p|_S)=B^2$ and $\operatorname{deg}(q|_S)=\Gamma^2$.

\item We have $$c^2_1(S)=c^2_1(X)B^2+c_1^2(Y)\Gamma^2+8c(\Gamma,B)-4 \Gamma^2B^2$$ and $$c_2(S)=c_2(X) B^2 + c_2(Y) \Gamma^2 + 4c(\Gamma,B) + 4 \Gamma^2 B^2$$ where $$c(\Gamma,B)=\frac{7}{2}\Gamma^2B^2+\frac{3}{2}(\Gamma \cdot K_X)B^2+\frac{3}{2}(B \cdot K_Y)\Gamma^2+\frac{1}{2}(\Gamma \cdot K_X)(B \cdot K_Y).$$
   
\item $K_S$ is big and nef.
\end{enumerate}
\end{theorem}

\begin{proof}
We first construct a surface $S\subset X\times Y$ which satisfies (1) and (2). Let $M:=p^*(\Gamma) \otimes q^*(B)$. Then, by Proposition \ref{Mlef}, we have that $M$ is lef with $\operatorname{exp}(M)=1$. We take general sections $E, E'$ of $M$, and we define $S:=E \cap E'$. We note that this intersection is nonempty and nonsingular by Bertini's theorem since $M$ is base point free and has enough sections. By Theorem \ref{bertini}, we have that $E$ is a nonsingular projective variety and $M|_E$ is lef with exp$(M|_E)=1$. Since $S=E'|_E$ is smooth, we have by \cite[Prop.2.1.5]{de2002hard} that $H^0(S,\Z)\simeq H^0(E,\Z)=\Z$, and so $S$ is a nonsingular projective surface. Moreover, by Corollary \ref{Flef}, we have that $\pi_1(S) \simeq \pi_1(X)\times\pi_1(Y)$. We also have that the degree of $p|_S$ is $((p^*(\Gamma) \otimes q^*(B))|_Y)^2=B^2$. Similarly the morphism $q|_S$ has degree $\Gamma^2$.
\bigskip

Now we prove (3). By the adjunction formula applied twice, and since $K_{X\times Y}\sim p^*(K_X)+ q^*(K_Y)$, we get that $$K_S\sim p|_S^*(K_X+2\Gamma)+q|_S^*(K_Y+2B).$$ To do the computation, we note that given nonsingular curves $C, C'$ in $X,Y$ respectively, we have $$p|_S^*(C) \cdot q|_S^*(C')=p^*(C) \cdot q^*(C') \cdot E \cdot E'=(C\times C') \cdot M^2=M|_{C \times C'}^2,$$ and so $p|_S^*(C) \cdot q|_S^*(C')=2(\Gamma \cdot C)(B \cdot C')$. This extends to find the intersection $p|_S^*(D) \cdot q|_S^*(D')$ for any divisors $D,D'$ in $X,Y$ respectively, and so \begin{align*}
 K_S^2 & =(p|_S^*(K_X+2\Gamma)+q|_S^*(K_Y+2B))^2 \\
 & =B^2(K_X+2\Gamma)^2 + \Gamma^2 (K_Y+2B)^2 \\ & + 4((K_X+2\Gamma) \cdot \Gamma)((K_Y+2B) \cdot B) \\
 & =K_X^2 B^2+K_Y^2 \Gamma^2+24 \Gamma^2B^2+12((\Gamma \cdot K_X) B^2+(B \cdot K_Y)\Gamma^2) \\
 & + 4(\Gamma \cdot K_X)(B \cdot K_Y).
\end{align*}

To calculate $\chi(S)$, we use the following exact Koszul complex. Since $S$ is a complete intersection of two sections of $M$ and $X \times Y$ is nonsingular, then we have (see e.g. \cite[Pages 76-77]{FL}) $$0\to\mathcal{O}_{X\times Y}(-2M)\to\mathcal{O}_{X\times Y}^{\oplus 2}(-M)\to\mathcal{O}_{X\times Y}\to\mathcal{O}_S\to 0.$$ 

By the additivity of the Euler characteristic and the K\"unneth formula (see e.g. \cite[Theo. 17.23]{Cut18}) $$H^n(X\times Y, M)=\displaystyle\bigoplus_{i+j=n}H^i(X,\Gamma)\otimes H^j(Y,B),$$ we obtain
\begin{align*}
    \chi(\mathcal{O}_S)&=\chi(\mathcal{O}_{X\times Y})
    +\chi(\mathcal{O}_{X\times Y}(-2A-2B))-2\chi(\mathcal{O}_{X\times Y}(-\Gamma-B))\\
    &=\chi(\mathcal{O}_X)\chi(\mathcal{O}_Y)+
    \chi(\mathcal{O}_X(-2\Gamma))\chi(\mathcal{O}_Y(-2B))\\
    &-2\chi(\mathcal{O}_X(-\Gamma))\chi(\mathcal{O}_Y(-B))\\
    &=\chi(\mathcal{O}_X)\chi(\mathcal{O}_Y)\\
    &+(\chi(\mathcal{O}_X)+\frac{1}{2}(4\Gamma^2+2(\Gamma \cdot K_X)))(\chi(\mathcal{O}_Y)+\frac{1}{2}(4B^2+2(B \cdot K_Y)))\\
    &-2(\chi(\mathcal{O}_X)+\frac{1}{2}(\Gamma^2+\Gamma \cdot K_X))
    (\chi(\mathcal{O}_Y)+\frac{1}{2}(B^2+B \cdot K_Y))\\
    &=\chi(\mathcal{O}_X)B^2+\chi(\mathcal{O}_Y)\Gamma^2+c(\Gamma,B),
   \end{align*}
where $$c(\Gamma,B)=\frac{7}{2}\Gamma^2B^2+\frac{3}{2}(\Gamma \cdot K_X)B^2+\frac{3}{2}(B \cdot K_Y)\Gamma^2+\frac{1}{2}(\Gamma \cdot K_X)(B \cdot K_Y).$$

Finally we show (4). Let $C$ be an irreducible curve on $S$. Let $a=\operatorname{deg}{p|_{C}}=a$ and $b=\operatorname{deg}q|_{C}$. Then, by the projection formula for generically finite morphisms, we have $$ C \cdot K_S =C \cdot p|_S^*(K_X+2\Gamma)+ C \cdot q|_S^*(K_Y+2B) \ \ \ \ \ \ \ \  \ \ \ \  $$ $$ = a \, p(C) \cdot (K_X+2\Gamma) + b \, q(C) \cdot (K_Y+2B).$$ We note that $K_X$, $K_Y$, and $\Gamma$ are nef, and $B$ is very ample, and so $C \cdot K_S \geq 0$. Using the formula for $K_S^2$ above and by the same previous reasons, we obtain $K_S^2>0$.   
\end{proof}

We now present our main result, which puts together all the ingredients elaborated until now.

\begin{theorem}\label{densityg}
Let $Y$ be a nonsingular projective surface with $K_Y$ nef, and let $r \in [1,3]$ be a real number. Then there 
are minimal nonsingular projective surfaces $S$ with $c_1^2(S)/c_2(S)$ arbitrarily close to $r$, and $\pi_1(S) \simeq \pi_1(Y)$. 
\end{theorem}

\begin{proof} 
Let $X_p$ be the collection of simply connected surfaces described in Section \ref{s2}. Let $\Gamma_p$ be the line bundle defined in Proposition \ref{rhcc}. For any $p$ we have that  $\Gamma_p$ is lef by Proposition \ref{Llef}. (We note that $\Gamma_p$ is not ample because of the resolution of singularities involved in the construction of the surfaces $X_p$.) Let $B$ be a very ample divisor on $Y$. Note that we satisfy all the hypothesis in Theorem \ref{ct} with $X:=X_p$ and $\Gamma:=\Gamma_p$. Therefore, there are surfaces $S_p:=S$ such that all the conclusions in Theorem \ref{ct} hold. In particular, we have $\pi_1(S_p) \simeq \pi_1(Y)$. 

The formulas in Theorem \ref{ct} part (3) are $$c^2_1(S_p)=c^2_1(X_p)B^2+c_1^2(Y)\Gamma_p^2+8c(\Gamma_p,B)-4\Gamma_p^2B^2$$ and $$c_2(S_p)=c_2(X_p) B^2 + c_2(Y) \Gamma_p^2 + 4c(\Gamma_p,B) + 4 \Gamma_p^2 B^2,$$ where $c(\Gamma_p,B)$ is as in Theorem \ref{ct}.

By Proposition \ref{rhcc} we have that $\Gamma_p^2=p$ and $\Gamma_p.K_{X_p}$ is a polynomial in $p$ of degree $3$. Thus $c(\Gamma_p,B)$ is a polynomial in $p$ of degree $3$. By Section \ref{s2}, the invariants $c^2_1(X_p)$ and $c_2(X_p)$ are Laurent polynomials in $p$ of degree $5$. Therefore, by the formulas above, we have  $$ \text{lim}_{p \to \infty} \ \frac{c_1^2(S_p)}{c_2(S_p)} = \text{lim}_{p \to \infty} \ \frac{c_1^2(X_p)}{c_2(X_p)} = \frac{27x^4+48x^2+8}{9x^4+48x^2+8}=:\lambda(x)$$ where $x:=\alpha/\beta$, as in Section \ref{s2}. In this way, just as in \cite[Thereom 6.3]{RU15}, we obtain the desired surfaces $S=S_p$ with $c_1^2(S)/c_2(S)$ arbitrarily close to $r$.
\end{proof}

\begin{corollary} \label{generaldensity}
Let $G$ be the fundamental group of a nonsingular projective surface. Then the Chern slopes $c_1^2(S)/c_2(S)$ of nonsingular projective surfaces $S$ with $\pi_1(S) \simeq G$ are dense in $[1,3]$.
\end{corollary}

\begin{proof}
Since $\pi_1$ is invariant under birational transformations between nonsingular projective surfaces, then it is enough to consider surfaces with no $(-1)$-curves. If $G$ is the fundamental group of $\P^1 \times C$, where $C$ is a nonsingular projective curve, then, for example, we can take as $Y$ a surface in \cite[Corollary 6.4]{RU15} to apply Theorem \ref{densityg}. Otherwise, we have a non-ruled surface with nef canonical class, and we can directly use Theorem \ref{densityg}.
\end{proof}

As we remarked in the introduction, the previous corollary involves the fundamental group $G$ of any nonsingular projective variety by means of the usual Lefschetz hyperplane theorem.

One may be tempted to use the result of Persson \cite{P81} on density of Chern slopes of simply connected minimal surfaces of general type in $[1/5,2]$ as an imput in Theorem \ref{densityg}, but the strategy does not work. It is not clear in that case how to find a suitable line bundle $\Gamma_p$ which makes things work. On the top of that, and as it was said in the introduction, this cannot work in full generality since, for example, from \cite{MP07} one can deduce that: If $S$ is a surface of general type with $c_1^2(S) < \frac{1}{3} c_2(S)$ and $\pi_1(S)$ finite, then the order of $\pi_1(S)$ is at most $9$. In this way, the question of ``freedom" of fundamental groups remains open for the interval $[1/3,1]$. 

We finish with two conjectures in relation to geography of Chern slopes for surfaces with ample canonical class, and for Brody hyperbolic surfaces. They could be proved through the theorems in this section if we can show that the projection $$q|_{S_p} \colon S_p \to Y$$ is a finite morphism (see Theorem \ref{ct}). This depends on the line bundles $\Gamma_p$. Catanese proves in \cite[Lemma 1.1]{catanese2000fibred} that $q|_{S_p}$ is a finite morphism if $\Gamma_p$ is very ample. We note that in \cite{RU15} it is proved that Chern slopes $c_1^2/c_2$ of simply connected minimal surfaces of general type are dense in $[1,3]$, but canonical class for all the constructed surfaces was not ample, because of the presence of arbitrarily many $(-2)$-curves. 

\begin{conjecture}
Let $G$ be the (topological) fundamental group of a nonsingular complex projective surface. Then Chern slopes $c_1^2(S)/c_2(S)$ of minimal nonsingular projective surfaces of general type $S$ with $\pi_1(S)$ isomorphic to $G$ and ample canonical class are dense in $[1,3]$.
\end{conjecture}

\begin{conjecture}
Let $Y$ be a Brody hyperbolic nonsingular projective surface. Then Chern slopes of hyperbolic nonsingular projective surfaces $S$ with $\pi_1(S)$ isomorphic to $\pi_1(Y)$ are dense in $[1,3]$.
\end{conjecture} 

%------------------------------------------------------


\begin{thebibliography}{99}

\bibitem[ABCKT96]{ABCKT96} 
	J. Amor\'os, M. Burger, K. Corlette, D. Kotschick, D. Toledo,
	\emph{Fundamental groups of compact Kaehler manifolds}, 			
	Mathematical Surveys and Monographs, 44. American Mathematical Society, Providence, RI, 1996.


\bibitem[A95]{A95} 
	D. Arapura,
	\emph{Fundamental groups of smooth projective  
	varieties}, Current topics in complex algebraic 	
	geometry (Berkeley, CA, 1992/93), 1–16, 
	Math. Sci. Res. Inst. Publ., 28, Cambridge Univ. 	
	Press, Cambridge, 1995. 

\bibitem[AN10]{AN10} 
	M. Aprodu, J. Nagel, 
	\emph{Koszul cohomology and algebraic geometry}, University 		
	Lecture Series, 52. American Mathematical Society, Providence, RI, 2010.

\bibitem[BK06]{BK06} 
	S. Baldridge, P. Kirk, 
	\emph{Symplectic 4-manifolds with arbitrary fundamental group near the Bogomolov--Miyaoka--Yau line},	
	J. Symplectic Geom. 4, Number 1 (2006), 63--70.
	
\bibitem[BK07]{BK07} 
	S. Baldridge, P. Kirk, 
	\emph{On symplectic 4-manifolds with prescribed fundamental group},	Comment. Math. Helv. 82, No. 4, 845--875 (2007).
	
\bibitem[BHPV04]{BHPV04}
    W. P. Barth, K. Hulek, C. A. M. Peters, A. Van de Ven,
    \emph{Compact complex surfaces},
    Ergebnisse der Mathematik und ihrer Grenzgebiete. 3. Folge., second edition, vol. 4, Springer-Verlag, Berlin, 2004.

		 
\bibitem[Bea96]{beauville1996complex} 
	A. Beauville,
	\emph{Complex algebraic surfaces},
	Cambridge University Press, Vol. 34, 1996.

\bibitem[Cat00]{catanese2000fibred}
	F. Catanese,
	\emph{Fibred surfaces, varieties isogenous to a product and 
	related moduli spaces},
	Amer. J. Math. 122 (2000), no. 1, 1--44. 

\bibitem[CM02]{de2002hard}
	M. de Cataldo, L. Migliorini,
	\emph{The hard Lefschetz theorem and the topology of semismall 	maps},
	Ann. Sci. École Norm. Sup. (4) 35 (2002), no. 5, 759--772. 

\bibitem[Ch87]{Ch87}
	Z. J. Chen,
	\emph{On the geography of surfaces. Simply connected minimal surfaces with positive index}, 
	Math. Ann. 277 (1987), no. 1, 141--164. 

\bibitem[Cut18]{Cut18}
	S. Cutkosky,
	\emph{Introduction to Algebraic Geometry}, 
	American Mathematical Society, Providence, RI, 2018.
	
\bibitem[D79]{D79}
	P. Deligne,
	\emph{Le groupe fondamental du complement d'une courbe plane}, 
	Seminaire Bourbaki 543 (Nov. 1979), Springer Lecture Notes in Mathematics, vol. 842, Springer-Verlag, N.Y. (1981).

%\bibitem[Esn17]{Esn17}
%	H. Esnault,
%	\emph{Survey on some aspects of Lefschetz theorems in algebraic 
%	geometry},
%	Rev. Mat. Complut. 30 (2017), no. 2, 217–232.
	
%\bibitem[Eyr04]{Eyr04}
%	C. Eyral,
%	\emph{Topology of quasi-projective varieties and Lefschetz 
%	theory},
%	New Zealand J. Math. 33 (2004), no. 1, 63–81.

\bibitem[FL85]{FL}
	W. Fulton, S. Lang,
	\emph{Riemann-Roch algebra},
	Grundlehren der Mathematischen Wissenschaften, 277. Springer-Verlag, New York, 1985.
	
\bibitem[G77]{G77}
	D. Gieseker,
	\emph{Global moduli for surfaces of general type},
	Invent. Math. 43(1977), 233--282.
	
\bibitem[G95]{G95}
	R. E. Gompf,
	\emph{A new construction of symplectic manifolds},
	Ann. of Math. (2) 142 (1995), no. 3, 527--595.

\bibitem[GM88]{goresky1988stratified}
	M. Goresky, R.  MacPherson,
	\emph{Stratified Morse Theory},
	Ergebnisse der Mathematik und ihrer Grenzgebiete (3) [Results in 		Mathematics and Related Areas (3)], 14. Springer-Verlag, Berlin, 1988.
	
\bibitem[Ho75]{Ho75}
	E. Horikawa,
	\emph{On deformations of quintic surfaces}, 
	Invent. Math. 31 (1975), no. 1, 43--85.
	
\bibitem[Ho76]{Ho76}
	E. Horikawa,
	\emph{Algebraic surfaces of general type with small $c_1^2$}, I Ann. of Math. (2) 104 (1976), no. 2, 357--387. II Invent. Math. 37 (1976), no. 2, 121--155. III  Invent. Math. 47 (1978), no. 3, 209--248. IV Invent. Math. 50 (1978/79), no. 2, 103--128. V J. Fac. Sci. Univ. Tokyo Sect. IA Math. 28 (1981), no. 3, 745--755. 
	
\bibitem[KK14]{KK14}
	M. Kapovich, J. Koll\'ar,
	\emph{Fundamental groups of links of isolated singularities}, 
	J. Amer. Math. Soc. 27 (2014), no. 4, 929--952. 

\bibitem[K13]{K13}
	M. Kapovish,
	\emph{Dirichlet fundamental domains and topology of projective varieties},
	Invent. Math. 194 (2013), no. 3, 631--672. 

\bibitem[KL09]{KL09}
	P. Kirk, Ch. Livingston,
	\emph{The geography problem for 4-manifolds with specified fundamental group},
	Trans. Amer. Math. Soc. 361(2009), Number 8, 4091--4124.
		
\bibitem[Laz17]{Laz17}
 	R. Lazarsfeld,
	\emph{Positivity in algebraic geometry. I. Classical setting: 		line bundles and linear series},  Ergebnisse der Mathematik und ihrer Grenzgebiete. 3. Folge. A Series of Modern Surveys in Mathematics, 48. Springer-Verlag, Berlin, 2004.
	
\bibitem[MP07]{MP07}
	M. Mendes Lopes, R. Pardini,
	\emph{On the algebraic fundamental group of surfaces with $K^2 \leq 3 \chi$},
	J. Differential Geom. 77 (2007), no. 2, 189--199. 

\bibitem[Mig95]{Mig95}
	L. Migliorini,
	\emph{A smooth family of minimal surfaces of general type over a curve of genus at most one is trivial},
	J. Algebraic Geom. 4 (1995), no. 2, 353--361.

\bibitem[Mi84]{Mi84}
	Y. Miyaoka,
	\emph{The maximal number of quotient singularities on surfaces with given numerical invariants},
	Math. Ann. 268(1984), no. 2, 159--171. 
	
\bibitem[Par05]{Par05}
	R. Pardini,
	\emph{The Severi inequality $K^2 \geq 4 \chi$ for surfaces of maximal Albanese dimension},
	Invent. Math. 159(2005), issue 3, 669--672.	
	
\bibitem[Park07]{Park07}
	J. Park,
	\emph{The geography of symplectic 4-manifolds with an arbitrary fundamental group},
	Proc. Am. Math. Soc. 135, No. 7, 2301--2307 (2007). 

\bibitem[P81]{P81}
	U. Persson,
	\emph{Chern invariants of surfaces of general type},
	Compositio Math. 43 (1981), no. 1, 3--58. 
	
\bibitem[PPX96]{PPX96}
    U. Persson, C. Peters, G. Xiao,
    \emph{Geography of spin surfaces},
    Topology 35 (1996), no. 4, 845--862.

\bibitem[RU15]{RU15}
	X. Roulleau, G. Urz\'ua,
	\emph{Chern slopes of simply connected complex surfaces 	of general type are dense in [2,3]},
	Ann. of Math. (2) 182 (2015), no. 1, 287--306.
	
\bibitem[S58]{S58}
	J.-P. Serre,
	\emph{Sur la topologie des vari\'et\'es alg\'ebriques 		en charact\'eristique p},
	Symposium Internacional de Topolog\'ia Algebraica, Universidad Nacional Aut\'onoma de M\'exico, 1958, pp. 24--53. 
	
\bibitem[U10]{U10}
    G. Urz\'ua,
    \emph{Arrangements of curves and algebraic surfaces},
    J. of Algebraic Geom. 19 (2010), 335--365.	

\bibitem[U17]{U17}
	G. Urz\'ua, 
	\emph{Chern slopes of surfaces of general type in positive characteristic},
	Duke Math. J. 166(2017), Number 5, 975--1004.

\bibitem[Yau77]{Yau77}
    S.-T. Yau,
    \emph{Calabi's conjecture and some new results in algebraic geometry},
    Proc. Nat. Acad. Sci. U.S.A. 74 (1977), 1798-1799.
	
\bibitem[X87]{X87}
	G. Xiao,
	\emph{Fibered algebraic surfaces with low slope},
	Math. Ann. 276(1987), no.3, 449--466. 


\end{thebibliography}
\end{document}